\documentclass[12pt,english]{amsart}

\input xy
\xyoption{all}

\usepackage[latin1]{inputenc}
\usepackage{amsfonts,amsmath,amssymb,amsthm}

\newcommand{\Pp}{\mathbb{P}}

\newcommand{\Zz}{\mathbb{Z}}
\newcommand{\Qq}{\mathbb{Q}} 
\newcommand{\Ff}{\mathbb{F}}

{\theoremstyle{plain}
\newtheorem{theorem}{Theorem}[section]    

\newtheorem{twisting lemma}[theorem]{Twisting lemma}

\newtheorem{lemma}[theorem]{Lemma}       
\newtheorem{corollary}[theorem]{Corollary}   

}
{\theoremstyle{remark}
\newtheorem{remark}[theorem]{Remark}   

}

\def\sep{{\scriptsize\hbox{\rm sep}}}

\def\Gabs{\hbox{\rm G}}

\def\Gal{\hbox{\rm Gal}}

\def\cm{\hbox{\hbox{\rm C}\kern-5pt{\raise 1pt\hbox{$|$}}}}

\def\lhfl#1#2{\smash{\mathop{\hbox to 12mm{\leftarrowfill}}
\limits^{#1}_{#2}}}

\def\rhfl#1#2{\smash{\mathop{\hbox to 12mm{\rightarrowfill}}
\limits^{#1}_{#2}}}

\def\build#1_#2^#3{\mathrel{
\mathop{\kern 0pt#1}\limits_{#2}^{#3}}}

\def\htrait#1#2{\smash{\mathop{\hbox to 12mm{\hrulefill}}
\limits^{#1}_{#2}}}

\def\sxbullet{{\raise 2pt\hbox{\bf .}}}

\begin{document}

\title[Twisted covers and specializations]{Twisted covers and specializations}

\author{Pierre D\`ebes}

\author{Fran\c cois Legrand}

\email{Pierre.Debes@math.univ-lille1.fr}

\email{Francois.Legrand@math.univ-lille1.fr}

\address{Laboratoire Paul Painlev\'e, Math\'ematiques, Universit\'e Lille 1, 59655 Villeneuve d'Ascq Cedex, France}



\date{\today}

\begin{abstract} 
The central topic is this question: is a given $k$-\'etale algebra $\prod_lE_l/k$ the specialization 
of a given $k$-cover $f:X\rightarrow B$ at some point $t_0\in B(k)$? Our main tool is a {\it twisting lemma} that reduces the problem to finding $k$-rational points on a certain $k$-variety. Previous forms of this twisting lemma are generalized and unified. New applications are given: a Grunwald form of Hilbert's irreducibility theorem over number fields, a non-Galois variant of the Tchebotarev theorem for function fields over finite fields, some general specialization properties of covers over PAC or ample fields. 
\end{abstract}

\maketitle



\section{Presentation}\label{sec:introduction}

\subsection{The central question} If $f:X\rightarrow B$ is an algebraic cover defined over a field $k$ and 
$t_0$ a $k$-rational point on $B$, not in the branch locus of $f$, the specialization of $f$ at $t_0$ is defined as a finite $k$-\'etale algebra of degree $n=\deg(f)$. For example, if $B=\Pp^1$ and $f$ is given by some
polynomial $P(T,Y)\in k[T,Y]$, it is the product of separable field extensions of $k$ that correspond to the 
irreducible factors of $P(t_0,Y)$ (for all but finitely many $t_0\in k$). Our central question is whether
a given degree $n$ $k$-\'etale algebra $\prod_{l} E_l/k$  is the specialization of some given degree $n$ $k$-cover $f:X\rightarrow B$ at some unramified point $t_0\in B(k)$. The classical Hilbert specialization 
property corresponds to the special case for which \'etale algebras are taken to be single degree $n$ field extensions and the answer is positive for at least one of them.

The question has already been investigated in \cite{DEGha} and \cite{DEGha2} for regular 
Galois covers and in \cite{DeLe1} for covers with geometric monodromy group $S_n$ (definitions recalled in \S \ref{ssec:basic2}).
The aim of this paper is to handle the situation of arbitrary covers, to provide a unifying 
approach and to give further 
applications. 


\subsection{The twisting lemma} Our main tool is a {\it twisting lemma} that gives a general answer to the question: under certain hypotheses, {\it answer is {\rm Yes} if there exist unramified $k$-rational points on the covering space 
$\widetilde X$ of certain twisted covers $\widetilde f:\widetilde X\rightarrow B$}. This lemma has several variants. The first one, for regular Galois covers,  was established in \cite{DeBB} for covers of $\Pp^1$ and in \cite{DEGha} for a general base space. It is used 
in \cite{DeLe1} to obtain the second one, for  covers with geometric monodromy group $S_n$.
We will prove the two variants shown on top row of the following diagram, which indicates that 
they generalize the two previous ones, shown on bottom row.

$$\begin{matrix}
&&& &\hbox{{\rm Galois}} & \Leftrightarrow & {\rm general}\\
&&& &  \Downarrow & & \Downarrow \\
&&&  &
\hbox{{\rm regular Galois}} &
\Rightarrow & {\rm monodromy}\ S_n \\
\end{matrix}$$

The {\it Galois} variant is for the situation $f:X\rightarrow B$ is a Galois cover, regular or not; it is proved in \S \ref{ssec:Galois_covers}. The {\it general} variant is proved in \S \ref{ssec:twisting-lemma-general-form} and concerns arbitrary covers, Galois or not, regular or not. Implication $\Rightarrow$ in upper 
row means that the general variant will be obtained from the Galois variant. 
We will also be interested in the converse in the twisting lemma: answer to
the original question is Yes {\it if and only if} there exist unramified $k$-rational points on the twisted covers.


The twisting lemma is a geometric {\it avatar} 
of an argument of Tchebotarev known as the {\it Field Crossing Argument} and which notably
appears in the proof of  the Tchebotarev density theorems over global fields and in the theory 
of PAC fields (see \cite{FrJa}).
The twisting lemma formalizes the core of the argument 
and produces a geometric tool: the variety $\widetilde X$.
This allows a unifying approach over an arbitrary base field: 
questions are reduced to finding rational points on $\widetilde X$. Letting 
the base field vary then yields previous results in various contexts and leads to new applications. The 
twisted cover $\widetilde f:\widetilde X\rightarrow B$, which appeared first 
in \cite{DeBB}, could also be defined by using the language of torsors.
Another related approach using an embedding problem presentation has also 
been recently proposed by Bary-Soroker \cite{Bary-Soroker_irreducible}.




\subsection{Applications}
As in previous papers, they are obtained over fields 
with good arithmetic properties: PAC fields, finite fields, number fields, ample fields. 
We present them below in connection with those from previous works.

 \subsubsection{} \hskip -1,5mm {\it Over a PAC field $k$} (definition recalled in \S \ref{ssec:field_PAC}), the regular Galois variant was used in \cite{DeBB} to prove that, given a group $G$ and a subgroup $H\subset G$, any Galois extension $E/k$ of group $H$ is a specialization of any regular Galois $k$-cover $f:X\rightarrow \Pp^1$ of group $G$ (thereby proving the so-called Beckmann-Black conjecture for PAC fields). 
A not necessarily Galois analog with an arbitrary degree $n$ $k$-\'etale algebra $\prod_{l} E_l/k$ replacing $E/k$ is proved in \cite{DeLe1} under the assumption that $f$ is a degree $n$ $k$-cover of geometric monodromy group $S_n$.  Corollary \ref{cor:PAC1} is a refinement of the first result (the regularity assumption is relaxed) while corollary \ref{cor:PAC2} is a variant of the second one (allowing more general monodromy groups). Similar applications have been obtained by Bary-Soroker \cite{Bary-Soroker_irreducible}.

The general spirit of these results is that over a PAC field there is no diophantine obstruction\footnote{in the sense that existence of rational points on some variety, which is a condition of our twisting lemma in general, is automatic over a PAC field $k$.} to a given \'etale algebra being a specialization of some given cover; obstructions only come from Galois theory. This has some impact on the arithmetic of PAC fields. For example a by-product of \cite{DeLe1} is that if $k$ is a PAC field of characteristic $0$
(for simplicity), every degree $n$ extension $E/k$ can be realized by some trinomial $Y^n - Y + b$ 
with $b\in k$.

\subsubsection{} \hskip -1,5mm {\it Over a finite field $k=\Ff_q$}, the twisting lemma can be combined with Lang-Weil to obtain an estimate for the number of points $t_0\in \Ff_q$ at which a given degree $n$ \'etale algebra $\prod_{l} E_l/\Ff_q$ is a specialization of a given degree $n$ $\Ff_q$-cover $f:X\rightarrow \Pp^1$ of geometric monodromy group $S_n$ (corollary \ref{cor:finite-fields}). This type of result is known in the literature as {\it Tchebotarev theorem for function fields over finite fields}. For example, if $\prod_{l} E_l/\Ff_q$ is the single degree $n$ field extension $\Ff_{q^n}/\Ff_q$, the estimate is of the form $q/n + O(\sqrt{q})$.
In the 
specific case where $f$ is given by the trinomial $Y^n+Y-T$, it yields results of Cohen 
and Ree proving a conjecture of Chowla. See \S \ref{ssec:finite_fields} for details and references.

For finite fields $\Ff_q$, the same general spirit as for PAC fields can be retained --- no diophantine obstruction to the problem ---, but provided that $q$ be suitably large.


\subsubsection{} \hskip -1,5mm {\it The local-global situation} of a number field $k$ given with some completions $k_v$ was central in \cite{DEGha}. The main result was a Hilbert-Grunwald theorem showing that every regular Galois $k$-cover $f:X\rightarrow \Pp^1$ of group $G$ has specializations at points $t_0\in k$ that are {\it Galois field extensions of group $G$} (Hilbert) with the extra property (Grunwald) that {\it they induce 
prescribed unramified extensions $E^v/k_v$ of Galois group $H_v\subset G$} at each finite place 
$v$ in a given finite set $S$, the only condition on the places being that the residue fields be suitably big and of order prime to $|G|$. An analog is given in \cite{DeLe1} for not necessarily 
Galois covers: the Hilbert condition becomes that the specialization at $t_0$ is {\it a degree $n$ field extension} and the Grunwald condition that {\it the local degrees} are imposed at each $v\in S$; this is proved  under the assumption that $f$ is a degree $n$ $k$-cover of geometric monodromy group $S_n$. 

\S \ref{ssec:number_fields} has a similar local-global flavor. The outcome is a generalization to general regular covers $f:X\rightarrow \Pp^1$ of the non-Galois analog above (corollary \ref{cor:effective}). On the way the following typical result of Fried is reproved (and generalized):  if the Galois group $\overline G\subset S_n$ over $\overline \Qq(T)$ of a degree $n$ polynomial $P(T,Y)\in \Qq(T)[Y]$ contains a $n$-cycle, then the associated Hilbert subset contains infinitely many arithmetic progressions with ratio a prime 
number. See \S \ref{ssec:number_fields} for details and references.

Here it is the relative flexibility of the local extensions obtained from global specializations that 
is the striking phenomenon. In the Galois situation, 
the very existence of global extensions with such local properties may sometimes even be questioned. 
Recall for example that results from \cite{DEGha} lead to some obstruction to the Regular Inverse Galois Problem (yet unproved to be not vacuous) related to some analytic questions around the Tchebotarev density theorem.

%

Other local-global situations can be considered, for example that of a base field that is a function field 
$\kappa(x)$ with $\kappa$ either a suitably large finite field or a PAC field with enough cyclic extensions. We refer to \cite{DEGha2} where these situations have been considered.

\subsubsection{} \hskip -1,5mm{\it Over ample fields} (definition recalled in \S \ref{ssec:ample_fields}),  
the twisting lemma leads to this general property 
of ample fields (corollary \ref{cor:ample}): if a $k$-cover 
$f:X\rightarrow B$ of curves specializes to some $k$-\'etale algebra $\prod_{l} E_l/k$ 
at some unramified point $t_0\in B(k)$, then it specializes to the same 
$k$-\'etale algebra $\prod_{l} E_l/k$ at infinitely many points $t\in B(k)$. 

\section{Basics} \label{sec:local_result}

Given a field $k$, we fix an algebraic closure $\overline k$ and denote the separable closure of $k$ in $\overline k$ by $k^\sep$ and its absolute Galois group by $\Gabs_k$. If $k^\prime$ is an overfield of $k$, we use the notation $\otimes_kk^\prime$ for the scalar extension from $k$ to $k^\prime$: for example, if $X$ is a $k$-curve,
$X\otimes_kk^\prime$ is the $k^\prime$-curve obtained by scalar extension. For more on this section, we refer to \cite[\S 2]{DeDo1} or \cite[chapitre 3]{coursM2}.

\subsection{Etale algebras and their Galois representations} \label{sssec:etale_algrebras}
Given a field $k$, a {\it $k$-\'etale algebra} is a product $\prod_{l=1}^s E_l/k$ of $k$-isomorphism classes of finite sub-field extensions $E_1/k, \ldots, E_s/k$ of $k^\sep/k$. Set $m_l=[E_l:k]$, $l=1,\ldots,s$ and $m=\sum_{l=1}^s m_l$. If $N/k$ is a Galois extension containing the Galois closures of $E_1/k, \ldots, E_s/k$, the Galois group $\Gal(N/k)$ acts by left multiplication on the left cosets of $\Gal(N/k)$ modulo $\Gal(N/E_l)$ for each $l=1,\ldots,s$. The resulting action $\Gal(N/k) \rightarrow S_m$ on all these left cosets, which is well-defined up to conjugation by elements of $S_m$, is called the {\it Galois representation of $\prod_{l=1}^s E_l/k$ relative to $N$}. Equivalently it can be defined as the 
action of $\Gal(N/k)$ on the set of all $k$-embeddings $E_l \hookrightarrow N$, $l=1,\ldots,s$.  

Conversely, an action $\mu: \Gal(N/k) \rightarrow S_m$ determines a $k$-\'etale algebra in the following way. For $i=1,\ldots,m$, denote the fixed field in $N$ of the subgroup of $\Gal(N/k)$ consisting of all $\tau$ such that $\mu(\tau)(i) = i$ by $E_i$. 
The product $\prod_lE_l/k$ for $l$ ranging over a set of representatives of the orbits 
of the action $\mu$ and where each extension $E_l/k$ is regarded modulo $k$-isomorphism is a $k$-\'etale algebra with $\sum_l [E_l:k] = m$. 

\vskip 1,5mm

\noindent
{\it {\rm G}-Galois variant}: if $\prod_{l=1}^s E_l/k$ is a {\it single Galois extension} $E/k$, the restriction
$\Gal(N/k)\rightarrow \Gal(E/k)$ is called {\it the {\rm G}-Galois representation of $E/k$} (relative to $N$). Any map $\varphi: \Gal(N/k) \rightarrow G$ obtained by composing $\Gal(N/k)\rightarrow \Gal(E/k)$ with a monomorphism $\Gal(E/k) \rightarrow G$ is called a G-{Galois} representation of $E/k$ 
(relative to $N$).
The extension $E/k$ can be recovered 
from $\varphi: \Gal(N/k) \rightarrow G$ by taking the fixed field in $N$ of ${\rm ker}(\varphi)$.  One obtains the Galois representation $\Gal(N/k)\rightarrow S_n$ of $E/k$ (relative to $N$) from a G-Galois representation $\varphi: \Gal(N/k) \rightarrow G$ (relative to $N$) by composing it
with the left-regular representation of the image group $\varphi(\Gal(N/k))$.


\subsection{Covers and function field extensions} \label{ssec:basic2}
Given a regular projective geometrically irreducible $k$-variety $B$, a {\it $k$-cover of $B$} is a finite and generically unramified morphism $f:X \rightarrow B$ defined over $k$ with $X$ a normal and
irreducible variety. Through the function field functor $k$-covers $f:X \rightarrow B$ correspond to finite separable field extensions $k(X)/k(B)$. The $k$-cover $f:X \rightarrow B$ is said to be {\it Galois} if the field extension $k(X)/k(B)$ is; if in addition $f:X\rightarrow B$ is given together with an
isomorphism $G\rightarrow \Gal(k(X)/k(B))$, it is called a $k$-G-Galois cover of group $G$. 

A $k$-cover $f:X\rightarrow B$ is said to be {\it regular} if $k(X)$ is a regular extension of $k$,
{\it i.e.} if $k(X)\cap \overline k=k$, or equivalently, if $X$ is geometrically irreducible. 
In general, there is some {\it constant extension} in $f:X\rightarrow B$, which we denote by $\widehat k_f/k$ and is defined by $\widehat k_f=k(X) \cap k^\sep$ (the special case $\widehat k_f=k$ corresponds to the situation  $f:X\rightarrow B$ is regular). 


If $f:X\rightarrow B$ is a $k$-cover, its Galois closure over $k$ is a Galois $k$-cover $g:Z \rightarrow B$, which {\it via} the cover-field extension dictionary, corresponds to the Galois closure of $k(X)/k(B)$. The Galois group $\Gal(k(Z)/k(B))$ is called the {\it monodromy group} of $f$. Denote next by $k^\sep(Z)$ the {\it compositum} of $k(Z)$ and $k^\sep$ (in a fixed separable closure of $k(B)$)\footnote{Note that as $g:Z\rightarrow B$ is Galois, $k(Z)$ only depends on the $k(B)$-isomorphism class of $k(X)/k(B)$ (but not on $k(X)/k(B)$ itself).}. The Galois group $\Gal(k^\sep(Z)/k^\sep(B))$ is called the {\it geometric monodromy group} of 
$f$; it is a normal subgroup of the monodromy group $\Gal(k(Z)/k(B))$. 
The {\it branch divisor} of the $k$-cover $f$ 
is 
the formal sum of all hypersurfaces of $B\otimes_k k^\sep$ such that the associated discrete 
valuations are ramified in the field extension $k^\sep(Z)/k^\sep(B)$. 

If $f:X\rightarrow B$ is regular, $f\otimes_kk^\sep$ is a $k^\sep$-cover, the Galois closure of its function field extension  is $k^\sep(Z)/k^\sep(B)$ and its branch divisor is the same as the branch divisor of $f$, and it is the formal sum of all hypersurfaces of $B\otimes_k k^\sep$ such that the associated discrete valuations are ramified in the field extension $k^\sep(X)/k^\sep(B)$.


\subsection{$\pi_1$-representations}
Given a reduced positive divisor $D\subset B$, denote the {\it $k$-fundamental group} of $B\setminus D$ by $\pi_1(B\setminus D, t)_k$ where $t\in B(\overline k)\setminus D$ is a base point (which corresponds to the choice of an algebraic closure of $k(B)$). Conjoining the two dictionaries covers-function field extensions and field extensions-Galois representations, we obtain the following correspondences: $k$-covers of $B$ of degree $n$ (resp. $k$-G-Galois covers of $B$ of group $G$) with branch divisor contained in $D$ correspond to transitive morphisms $\phi: \pi_1(B\setminus D, t)_k \rightarrow S_n$ 
(resp. to epimorphisms $\phi: \pi_1(B\setminus D, t)_k \rightarrow G$). 
The regularity property corresponds to the extra condition that the restriction of $\phi$ to $\pi_1(B\setminus D, t)_{k^\sep}$ remains transitive (resp. remains onto). 
These morphisms are called {\it fundamental group representations} ($\pi_1$-representations for short)
of the corresponding $k$-covers and $k$-G-Galois covers.

\subsection{Specializations} Each $k$-rational point $t_0\in B(k)\setminus D$ provides a section
${\rm s}_{t_0}: \Gabs_k\rightarrow \pi_1(B\setminus D, t)_k$ to the exact sequence

$$ 1\rightarrow \pi_1(B\setminus D, t)_{k^\sep} \rightarrow \pi_1(B\setminus D, t)_k \rightarrow \Gabs_k \rightarrow 1$$ 

\noindent
well-defined up to conjugation by elements in $\pi_1(B\setminus D, t)_{k^\sep}$. 

If $\phi: \pi_1(B\setminus D, t)_k \rightarrow G$ represents a $k$-G-Galois cover $f:X\rightarrow B$, the morphism $\phi \circ {\sf s}_{t_0}:\Gabs_k \rightarrow G$ is a G-Galois representation. The fixed field in $k^\sep$ of ${\rm ker}(\phi \circ {\sf s}_{t_0})$ is the residue field at some/any point above $t_0$ in the extension $k(X)/k(B)$. We denote it by $k(X)_{t_0}$ and call $k(X)_{t_0}/k$ {\it the specialization} of the $k$-G-Galois cover $f$ at $t_0$.

If $\phi: \pi_1(B\setminus D, t)_k \rightarrow S_n$ represents a $k$-cover $f:X\rightarrow B$, the morphism $\phi \circ {\sf s}_{t_0}:\Gabs_k \rightarrow S_n$ is the {\it specialization representation} 
of $f$ at $t_0$. The corresponding $k$-\'etale algebra is denoted by 
$\prod_{l=1}^s k(X)_{t_0,l}/k$ and called the {\it collection of specializations} of $f$ at $t_0$. Each field $k(X)_{t_0,l}$ is a residue extension at some prime above $t_0$ in the extension $k(X)/k(B)$ and {\it vice-versa}; $k(X)_{t_0,l}$ is called {\it a specialization} of $f$ at $t_0$. The {\it compositum} in $k^\sep$ of the Galois closures of all spe\-cia\-li\-zations at $t_0$ is {\it the} specialization at $t_0$ of the Galois closure of $f$ (viewed as a $k$-G-Galois cover). If the $k$-cover $f$ is regular, the fields $k(X)_{t_0,l}$ correspond to the definition fields of the points in the fiber $f^{-1}(t_0)$ and $\phi \circ {\sf s}_{t_0}:\Gabs_k \rightarrow S_n$ to the {\it action} of $\Gabs_k$ on them. 



\section{The twisting lemma} \label{sec:twisting}

Given a field $k$, the question we address is whether a given $k$-cover specializes to some given $k$-\'etale algebra at some unramified $k$-rational point. We first consider the situation of Galois covers in \S \ref{ssec:Galois_covers} and then handle the non-Galois situation 
in \S \ref{ssec:twisting-lemma-general-form} by ``going to the Galois closure''. The Galois situation was considered in \cite{DEGha} in the special case of {\it regular} Galois covers. But the Galois closure of a $k$-cover is not regular in general, even if $f:X\rightarrow B$ is regular, and this special case needs to be extended. \S \ref{ssec:Galois_covers} is a generalization of  the twisting lemma  from \cite{DEGha} to not necessarily regular Galois covers. 


\subsection{The twisting lemma for Galois covers} \label{ssec:Galois_covers}
Fix the field $k$ and a Galois $k$-cover $g:Z\rightarrow B$. Denote its branch divisor by $D$, the Galois group $\Gal(k(Z)/k(B))$ by $G$,  the $\pi_1$-representation of  the $k$-G-Galois cover $g: Z \rightarrow B$ by $\phi: \pi_1(B\setminus D, t)_k \rightarrow G$,
the geometric monodromy group $\Gal(k^\sep(Z)/k^\sep(B))$ by $\overline G$ and the {constant extension} in $g: Z \rightarrow B$ by $\widehat k_g/k$. 

\subsubsection{Twisting Galois covers} \label{ssec:twisting}
Let $N/k$ be some Galois extension with Galois group $H$ isomorphic to a subgroup of $G$.
With no loss we may and will view $H$ itself as a subgroup of $G$.
We assume the following {\it compatibility condition} of $N/k$ with the 
constant extension $\widehat k_g/k$:
\medskip

\noindent
({\rm const/comp}) {\it the fixed field $N^{H\cap \overline G}$ of $H \cap \overline G$ in $N$ is the field $\widehat k_g$}.
\medskip

\noindent
This condition is trivially satisfied in the regular case as both fields $N^{H\cap \overline G}$ and $\widehat k_g$ equal $k$.

Consider  the homomorphism $\Lambda: \Gabs_k \rightarrow G/\overline G$ induced by $\phi$ on the quotient 
$\Gabs_k = \pi_1(B\setminus D, t)_k/\pi_1(B\setminus D, t)_{k^\sep}$. The map $\Lambda$ is a G-Galois representation of the constant extension $\widehat k_g/k$ (relative to $k^{\sep}$); it is called the {\it constant extension map} \cite[\S 2.8]{DeDo1}.  As it is surjective, we have $\Gal(\widehat k_g/k) \simeq G/\overline G$ and so 
condition ({\rm const/comp})  implies that $H\overline G = G$.

Let  $\varphi: \Gabs_k \rightarrow H$ be the G-Galois representation of the Galois extension $N/k$ (relative to $k^{\sep}$) and $\overline{\varphi}: \Gabs_k \rightarrow G/\overline{G}$ be the composed map of $\varphi$ with the canonical surjection $\overline{\hskip 2pt \sxbullet \hskip 2pt}: 
G\rightarrow G/\overline G$. Hypothesis ({\rm const/comp}) rewrites as follows:

\vskip 4mm

\noindent
({\rm const/comp}) \hskip -1mm {\it There exists $\overline \chi \in {\rm Aut}(G/\overline G)$ such that $\Lambda = \overline \chi \circ \overline \varphi$.}
\vskip 3mm

\noindent (The equivalence follows from $\widehat k_g = (k^\sep)^{{\rm ker}(\Lambda)}$ and
\vskip 1mm

\centerline{$\displaystyle (k^\sep)^{{\rm ker}(\overline{\varphi})} = ((k^\sep)^{{\rm ker}(\varphi)})^{{\rm ker}(\overline{\varphi})/{\rm ker}(\varphi)} = N^{\varphi({\rm ker}(\overline \varphi))} = N^{H\cap \overline G}$\hskip 2mm.}
\vskip 1mm

\noindent
Also note that as $\Lambda:\Gabs_k \rightarrow G/\overline{G}$ is onto, an automorphism $\overline \chi$ satisfying ({\rm const/comp}) is unique).

Assume there exists an isomorphism $\chi:H\rightarrow H^\prime$ onto a subgroup $H^\prime \subset G$ that induces $\overline \chi$ modulo $\overline G$. 
With ${\rm Per}(G)$ the permutation group of $G$, consider then the map 

$$\widetilde \phi^{\chi \varphi}: \pi_1(B\setminus D, t)_k \rightarrow {\rm Per}(G)$$

\noindent
defined by this formula, where 
$r$ is the restriction $\pi_1(B\setminus D, t)_k \rightarrow \Gabs_k$:
for  $\theta\in \pi_1(B\setminus D, t)_k$ and $x\in G$,

$$ \widetilde \phi^{\chi \varphi}(\theta)(x)  =
\phi(\theta) \hskip 3pt x  \hskip 4pt (\chi \circ \varphi \circ r) (\theta)^{-1}$$

\noindent
It is easily checked that $\widetilde \phi^{\chi \varphi}$ is a group homomorphism.
However the corresponding action of $\pi_1(B\setminus D, t)_k$ on $G$ is not transitive in general.
More precisely we have the following.

\begin{lemma} \label{lem:appendix}
Under hypothesis {\rm (const/comp)}, we have $\widetilde \phi^{\chi \varphi}(\theta) (\overline G) \subset \overline G$ for every $\theta\in \pi_1(B\setminus D, t)_k$. 
\end{lemma}

\begin{proof}
For all $\theta\in \pi_1(B\setminus D, t)_k$ and $x\in \overline G$, we have:
\vskip 2mm

\hskip -2mm $\overline{\widetilde \phi^{\chi \varphi}(\theta)(x)} =\overline{\phi(\theta)} \hskip 2pt . \hskip 2pt\overline x   \hskip 1pt. \overline{(\chi \circ \varphi \circ r) (\theta)}^{-1} =  \Lambda(r (\theta)) \hskip 1pt. \hskip 1pt\overline \chi (\varphi(r(\theta)))^{-1}= 1$
\end{proof}

Consider the morphism, denoted by $\widetilde \phi^{\chi \varphi}_{\overline G}: \pi_1(B\setminus D, t)_k \rightarrow {\rm Per}(\overline G)$, that sends  $\theta\in \pi_1(B\setminus D, t)_k$ to the restriction of $\widetilde \phi^{\chi \varphi}(\theta)$ on $\overline G$. Its  restriction $\pi_1(B\setminus D, t)_{k^\sep} \rightarrow {\rm Per}(\overline G)$ is given by

$$ \widetilde \phi^{\chi \varphi}(\theta)(x)  =
\phi(\theta) \hskip 3pt x  \hskip 7mm  (\theta\in \pi_1(B\setminus D, t)_{k^\sep}, x\in \overline G)$$

\noindent
Thus this restriction is obtained by composing the original $\pi_1$-re\-pre\-sen\-tation $\phi$ restricted to $\pi_1(B\setminus D, t)_{k^\sep}$ with the left-regular representation $\overline G\rightarrow {\rm Per}(\overline G)$ of $\overline G$. This shows that $\widetilde \phi^{\chi \varphi}_{\overline G}: \pi_1(B\setminus D, t)_k \rightarrow {\rm Per}(\overline G)$ is the $\pi_1$-representation of some regular $k$-cover, which we denote by $\widetilde g^{\chi \varphi}: \widetilde Z^{\chi \varphi} \rightarrow B$ 
and call the {\it twisted cover} of $g$ by $\chi \varphi$.

\subsubsection{Statement of the twisting lemma for Galois covers} \label{ssec:statement_twisting-lemma}
The following statement gives the main property of the twisted cover.

Some notation is needed.  Conjugation automorphisms in some group ${\mathcal G}$ are denoted by ${\rm conj}(\omega)$ for $\omega\in {\mathcal G}$: ${\rm conj}(\omega)(x) = \omega \hskip 2pt x \hskip 2pt \omega^{-1}$ ($x\in {\mathcal G}$). The set of all isomorphisms $\chi:H\rightarrow H^\prime$ onto a subgroup $H^\prime \subset G$ that induce $\overline{\chi}$ modulo $\overline G$ is denoted by 
${\rm Isom}_{\overline \chi}(H,H^\prime)$. 

Fix then a set $\{\chi_\gamma:H\rightarrow H_\gamma \hskip 2pt | \hskip 2pt \gamma \in \Gamma\}$ of representatives of all isomorphims $\chi \in {\rm Isom}_{\overline \chi}(H,H^\prime)$ with $H^\prime$ ranging over all subgroups of $G$ isomorphic to $H$, modulo the equivalence 
that identifies $\chi_1 \in {\rm Isom}_{\overline \chi}(H,H^\prime_1)$ and $\chi_2 \in {\rm Isom}_{\overline \chi}(H,H^\prime_2)$ if $H^\prime_2=\omega \hskip 1pt H^\prime_1 \hskip 1pt \omega^{-1}$ and $\chi_2 \chi_1^{-1} = {\rm conj}(\omega)$ for some $\omega \in \overline G$.



\begin{twisting lemma}[Galois form] \label{prop:general_twisted_cover} 
Under condition {\rm (const/comp)}, we have the following conclusions {\rm (a)} and {\rm (b)}. 

\vskip 1,4mm
\noindent
{\rm (a)} For each subgroup $H^\prime\subset G$ isomorphic to $H$, each $\chi \in {\rm Isom}_{\overline \chi}(H,H^\prime)$
and each $t_0\in B(k)\setminus D$, these conditions are equivalent:

\vskip 0,4mm

\noindent
{\rm (i)} there exists a point $x_0\in \widetilde Z^{\chi \varphi} (k)$ such that 
$\widetilde g^{\chi \varphi} (x_0)=t_0$,

\vskip 1mm

\noindent
{\rm (ii)} there is $\omega \in \overline G$ such that $(\phi \circ {\sf s}_{t_0})(\tau) = 
\omega \hskip 2pt (\chi \circ \varphi) (\tau) 
\hskip 2pt \omega^{-1}$, $\tau \in \Gabs_k$, 
(where ${\sf s}_{t_0}: \Gabs_k \rightarrow \pi_1(B\setminus D, t)_k$ is the section
associated with $t_0$).

\vskip 3mm

\noindent
{\rm (b)} For each $t_0\in B(k)\setminus D$, 
the following are equivalent:
\vskip 0,8mm

\noindent
{\rm (iii)} the specialization $k(Z)_{t_0}/k$ of the $k$-{\rm G}-Galois cover $g:Z\rightarrow B$ is the extension $N/k$, 
\vskip 0,6mm

\noindent
{\rm (iv)}  there exists an isomorphism $\chi \in {\rm Isom}_{\overline \chi}(H,\phi \circ {\sf s}_{t_0}(\Gabs_k))$
such that conditions {\rm (i)}-{\rm (ii)}  hold for this $\chi$,



\vskip 0,6mm

\noindent
{\rm (v)} there exists $\gamma \in \Gamma$ such that conditions {\rm (i)}-{\rm (ii)} hold for $\chi=\chi_\gamma$.
\vskip 0,8mm

\noindent 
Furthermore an element $\gamma \in \Gamma$ as in {\rm (v)} is necessarily unique.
%
%

\end{twisting lemma}

A {\it single} twisted cover is involved in (a) while there are several in (b).
In this respect the representation viewpoint used in (a) may look more natural than the field extension one in (b). The latter however is more useful in practice. Also note that conditions (iv)-(v), being equivalent to (iii), do not depend on the chosen $\pi_1$-representation $\phi: \pi_1(B\setminus D, t)_k \rightarrow G$ of $g:Z\rightarrow B$ modulo conjugation by elements of $G$.

\begin{remark} \label{rem:after_twisting_non-regular_galois}
(a) Existence of some subgroup $H^\prime \subset G$ such that  the set ${\rm Isom}_{\overline \chi}(H,H^\prime)$ is non-empty, which amounts to $\Gamma \not= \emptyset$, is not guaranteed;
in this case conditions (iii)-(iv)-(v) fail.
It is however guaranteed under each of the assumptions $\overline \chi = {\rm Id}_{G/\overline G}$ or ${\rm Out}(G/\overline G)=\{1\}$. Indeed if $\overline \chi = {\rm Id}_{G/\overline G}$, then ${\rm Id}_H\in {\rm Isom}_{\overline \chi}(H,H)$, and if ${\rm Out}(G/\overline G)=\{1\}$, {\it i.e.},  if every automorphism of $G/\overline G$ is inner, then, as $H\overline G = G$, every inner automorphism ${\rm conj}(\overline \omega)$ with $\overline \omega \in G/\overline G$ lifts to some isomorphism 
${\rm conj}(\omega): H\rightarrow H$ with $\omega \in H$. Both assumptions includes the  regular case as then $G/\overline G=\{1\}$.
\vskip 1,5mm

\noindent
(b) Some uniqueness property can be added to (iv), as in (v). Indeed an isomorphism 
$\chi \in {\rm Isom}_{\overline \chi}(H,\phi \circ {\sf s}_{t_0}(\Gabs_k))$ satisfying conditions 
(i)-(ii), as the one in (iv), is necessarily unique up to left composition by ${\rm conj}(\omega)$ 
with $\omega \in {\rm Nor}_{\overline G}(\phi \circ {\sf s}_{t_0}(\Gabs_k))$. The advantage
of condition (v) is that the set $\bigcup_{\gamma\in \Gamma}\widetilde Z^{\chi_\gamma \varphi}(k)$
where unramified $k$-rational points should be found to conclude that (iii) holds does
not depend on $t_0$ (although the element $\gamma \in \Gamma$ in (v) does). Moreover the 
uniqueness property in (v) makes it easier to count the points $t_0 \in B(k)$ for which (iii) holds.
\vskip 1,5mm

\noindent
(c) The proof of (i) $\Leftrightarrow$ (ii) below shows further that the number of $k$-rational points on $\widetilde Z^{\chi \varphi}$ above some given unramified point $t_0\in B(k)$, if positive, is equal to the order of the group ${\rm Cen}_{\overline G}(\chi(H))$. 
\end{remark}

\subsubsection{Proof of the twisting lemma \ref{prop:general_twisted_cover}}
(a) Fix  a subgroup $H^\prime\subset G$ isomorphic to $H$, an isomorphism $\chi \in {\rm Isom}_{\overline \chi}(H,H^\prime)$ and a point $t_0\in B(k)\setminus D$.
The map $\widetilde \phi^{\chi \varphi}_{\overline G} \circ {\sf s}_{t_0}: \Gabs_k \rightarrow {\rm Per}(\overline G)$ 
is the action of $\Gabs_k$ on the fiber $(\widetilde g^{\chi \varphi})^{-1}(t_0)$; it is given by

$$ \widetilde \phi^{\chi \varphi}_{\overline G}({\sf s}_{t_0}(\tau)) = \phi({\sf s}_{t_0}(\tau)) \hskip 3pt  x \hskip 4pt  (\chi \circ \varphi) (\tau)^{-1} \hskip 6mm (\tau \in \Gabs_k, x\in \overline G)$$

\noindent
Elements 
$\widetilde \phi^{\chi \varphi}_{\overline G}({\sf s}_{t_0}(\tau))$ have a common fixed point $\omega \in \overline G$ if and only if  $\phi({\sf s}_{t_0}(\tau)) = \omega \hskip 2pt (\chi \circ \varphi) (\tau) 
\hskip 2pt \omega^{-1}$ ($\tau \in \Gabs_k$). This yields (i) $\Leftrightarrow$ (ii). Furthermore, the set of all $\omega \in \overline G$ satisfying the preceding condition, if non empty, is a left coset $\omega_0 \hskip 2pt {\rm Cen}_{\overline G}(\chi(H))$; this proves remark \ref{rem:after_twisting_non-regular_galois} (c).
\vskip 2,5mm

\noindent
(b) Fix $t_0\in B(k)\setminus D$ and a representative of the section ${\sf s}_{t_0}: \Gabs_k \rightarrow \pi_1(B\setminus D, t)_k$ (defined up to conjugation in $\pi_1(B\setminus D, t)_{k^\sep}$).
\vskip 1mm 

Implication (iv) $\Rightarrow$ (iii) follows from the fact that if $\chi \in {\rm Isom}_{\overline \chi}(H,\phi \circ {\sf s}_{t_0}(\Gabs_k))$ satisfies (i)-(ii), then ${\rm ker}(\phi \circ {\sf s}_{t_0})$ and ${\rm ker}(\varphi)$ are equal, and so so are their fixed fields in $k^\sep$. Conversely assume that the extensions $k(Z)_{t_0}/k$ and $N/k$ are equal, {\it i.e.} ${\rm ker}(\phi \circ {\sf s}_{t_0})$ and ${\rm ker}(\varphi)$ are the same subgroup, say ${\mathcal K}$, of $\Gabs_k$. The two morphisms $\phi \circ {\sf s}_{t_0}:\Gabs_k \rightarrow \phi \circ {\sf s}_{t_0}(\Gabs_k)\subset G$ and $\varphi:\Gabs_k \rightarrow H\subset G$ then differ from $\Gabs_k \rightarrow \Gabs_k/{\mathcal K}$ by some isomorphisms $\phi \circ {\sf s}_{t_0}(\Gabs_k)\rightarrow \Gabs_k/{\mathcal K}$ and $H\rightarrow \Gabs_k/{\mathcal K}$, respectively. Thus they differ from one another by an isomorphism $\chi:H\rightarrow \phi \circ {\sf s}_{t_0}(\Gabs_k)$:  $\phi \circ {\sf s}_{t_0} = \chi \circ \varphi$. 
It follows from this and from uniqueness of $\overline \chi$ satisfying (const/comp) that $\chi$ automatically induces $\overline \chi$ modulo $\overline G$. Conclude that $\chi \in {\rm Isom}_{\overline \chi}(H,\phi \circ {\sf s}_{t_0}(\Gabs_k))$ and conditions (i)-(ii) hold for this $\chi$. 
\vskip 1mm

Assume (v) holds, {\it i.e.}, for some $\gamma \in \Gamma$, condition (i)-(ii) are satisfied for the isomorphism $\chi_\gamma:H\rightarrow H_\gamma$ and some $\omega \in \overline G$. It readily follows that $\chi = {\rm conj}(\omega) \circ \chi_\gamma$ also satisfies (ii) and is in  ${\rm Isom}_{\overline \chi}(H, \phi \circ {\sf s}_{t_0}(\Gabs_k))$.
This establishes (iv).
Conversely assume (iv) holds. Let  $\chi \in {\rm Isom}_{\overline \chi}(H,\phi \circ {\sf s}_{t_0}(\Gabs_k))$ 
be an isomorphism such that conditions (i)-(ii) hold, for some $\omega \in  \overline G$. 
There exist $\gamma \in \Gamma$ and $\omega^\prime \in \overline G$ such that $\chi = {\rm conj}(\omega^\prime) \circ \chi_\gamma$. It follows that 
condition {\rm (ii)} holds for $\chi_\gamma$ as well (with conjugation factor $\omega \omega^\prime$). Uniqueness of $\gamma \in \Gamma$ in condition (v) readily follows from condition (ii) and the definition of the set $\{\chi_\gamma \hskip 2pt | \hskip 2pt \gamma \in \Gamma\}$. $\square$


\subsection{The general form of the twisting lemma} \label{ssec:twisting-lemma-general-form}
We fix a degree $n$ $k$-cover $f:X\rightarrow B$ and a degree $n$ $k$-\'etale algebra $\prod_{l=1}^s E_l/k$ and the question we address is whether $\prod_{l=1}^s E_l/k$ {\it is 
the collection $\prod_{l} k(X)_{t_0,l}/k$ of specializations of $f:X\rightarrow B$ at some point $t_0\in B(k)$}.


\subsubsection{Statement of the result} \label{ssec:statement_twisting_general}
Denote the branch divisor of $f:X\rightarrow B$ by $D$, its Galois closure by $g:Z\rightarrow B$, the Galois group $\Gal(k(Z)/k(B))$ by $G$,  the $\pi_1$-representation of  the $k$-G-Galois cover $g: Z \rightarrow B$ by $\phi: \pi_1(B\setminus D, t)_k \rightarrow G$, the Galois representation of the field extension 
$k(X)/k(B)$ relative to $k(Z)$ by $\nu: G\rightarrow S_n$,  the geometric monodromy group $\Gal(k^\sep(Z)/k^\sep(B))$ by $\overline G$ and the {constant extension} in $g: Z \rightarrow B$ by $\widehat k_g/k$. 

Let $N/k$ be the {\it compositum} inside $k^\sep$ of the Galois closures of the extensions $E_l/k$,
$l=1,\ldots, s$, and $H=\Gal(N/k)$. A necessary condition for a positive answer to the question is
 that $N$ be the {\it compositum} inside $k^\sep$ of the Galois closures of the extensions $k(X)_{t_0,l}/k$. In particular, $H$ should be isomorphic to some subgroup of $G$. From now on we will assume it. With no loss we may then and will view $H$ as a subgroup of $G$.
Finally let $\varphi: \Gabs_k \rightarrow H$ be the G-Galois representation of $N/k$ relative to $k^\sep$ and $\mu: H\rightarrow S_n$ be the Galois representation of 
$\prod_{l=1}^s E_l/k$ relative to $N$.

Some further notation from \S \ref{ssec:Galois_covers} is retained. The constant extension compatibility condition (const/comp) determines a unique automorphism $\overline \chi$ of $G/\overline G$ (\S \ref{ssec:twisting}). The twisted cover $\widetilde g^{\chi \varphi}: \widetilde Z^{\chi \varphi} \rightarrow B$ is defined for every isomorphism $\chi:H\rightarrow H^\prime$ onto a subgroup $H^\prime \subset G$ inducing $\overline{\chi}$ modulo $\overline G$ (\S \ref{ssec:twisting}). The set 
of all such isomorphisms $\chi:H\rightarrow H^\prime$ is denoted by ${\rm Isom}_{\overline \chi}(H,H^\prime)$. The isomorphisms $\chi_\gamma:H\rightarrow H_\gamma$ ($\gamma\in \Gamma$) are defined in \S \ref{ssec:statement_twisting-lemma}.









 
\begin{twisting lemma}[general form] \label{prop:twisted cover} 
Let $f:X\rightarrow B$ be a  $k$-cover and $\prod_{l=1}^s E_l/k$ be a $k$-\'etale algebra as 
above. Assume further that condition {\rm (const/comp)} from  \S \ref{ssec:twisting} holds for the Galois closure 
$g:Z\rightarrow B$ of $f$. Then for each $t_0\in B(k)\setminus D$, the following conditions
are equivalent:
\vskip 1,5mm

\noindent
{\rm (i)} $\prod_{l} E_l/k$ is the collection $\prod_{l} k(X)_{t_0,l}/k$
of specializations of $f:X\rightarrow B$ at the point $t_0$.
\vskip 1mm

\noindent
{\rm (ii)}  there is a subgroup $H^\prime \subset G$ isomorphic to $H$ and an isomorphism $\chi \in {\rm Isom}_{\overline \chi}(H,H^\prime)$ such that  

\hskip 0mm {\rm 1.} there exists $x_0\in \widetilde Z^{\chi \varphi}(k)$  with $\widetilde g^{\chi \varphi}(x_0)=t_0$, and  
\vskip 1mm

\hskip 0mm {\rm 2.} there exists $\sigma \in S_n$  that $\nu \circ \chi (h) = \sigma \hskip 2pt \mu(h) \hskip 2pt  \sigma^{-1}$ for every $h\in H$.
\vskip 1mm

\noindent
Furthermore if {\rm (ii)} holds, it holds for some isomorphism $\chi_\gamma: H\rightarrow H_\gamma$ for some $\gamma \in \Gamma$ and the element $\gamma$ is then necessarily unique.

%
\end{twisting lemma}

\subsubsection{About condition {\rm (ii-2)}} \label{remark:twisting_general} We focus on condition {\rm (ii-2)} which is the group-theoretical part of condition (ii) (while condition (ii-1) is the diophantine part).
\vskip 1mm

We first note for later use that the number of $\gamma \in \Gamma$ for which condition {\rm (ii-2)} holds for $\chi=\chi_\gamma$, 
if positive, is equal to the number of distinct isomorphisms $\chi_\gamma, \chi_{\gamma^\prime}$
($\gamma, \gamma^\prime \in \Gamma$) such that the actions $\nu\circ \chi_\gamma: H\rightarrow S_n$ and $\nu\circ \chi_{\gamma^\prime}: H\rightarrow S_n$ are conjugate in $S_n$.
\vskip 1mm

Below we give three standard situations where condition (ii-2) holds.
\vskip 1mm

\noindent
(a) {\it geometric monodromy group $S_n$}: $G=\overline G = S_n$ as in \cite{DeLe1}.
Condition (const/comp) holds and $\nu:S_n\rightarrow S_n$ is the natural action: $\nu = {\rm Id}_{S_n}$. Condition $\nu \circ \chi_\gamma (h) = \sigma \hskip 2pt \mu(h) \hskip 2pt  \sigma^{-1}$ ($h\in H$) is satisfied with $\chi_\gamma$ the representative of the isomorphism $\mu: H\rightarrow S_n$ (and some $\sigma \in S_n$).


\vskip 2mm

\noindent
(b) {\it Galois situation}: $f:X\rightarrow  B$ is a Galois $k$-cover, $\prod_{l} E_l/k$ is a Galois 
field extension $E/k$ of group $H\subset G$ and $\Gamma \not= \emptyset$. Then $\nu$ is the left-regular representation $G\rightarrow {\rm Per}(G)$ and $\mu$ its restriction $H\rightarrow {\rm Per}(G)$.
Note next that if $\gamma \in \Gamma$, the restriction $\nu|_H: H\rightarrow {\rm Per}(G)$ and $\nu \circ \chi_\gamma:H\rightarrow {\rm Per}(G)$ are conjugate actions. Condition (ii-2) follows.
\vskip 2mm

In (c) below, the {\it type of a permutation} $\sigma \in S_n$ is the (multiplicative) divisor of all lengths of disjoint cycles involved in the cycle decomposition of $\sigma$ (for example, an $n$-cycle is of type $n^1$).
\vskip 1mm

\noindent
(c) {\it cyclic specializations}:  condition (const/comp) holds, $H$ is a cyclic subgroup of $G$ generated by an element $\omega$ such that $\nu(\omega)$ is of type equal to the divisor of all 
degrees $[E_l:k]$ of field extensions in the \'etale algebra $\prod_{l} E_l/k$. 
\vskip 0,7mm

Indeed for every integer $a\geq 1$ such that $(a,|H|)=1$, let $\chi_a: H\rightarrow H$ be the morphism that maps $\omega$ to $\omega^a$. As $H\overline G = G$, each map $\chi_a$ induces an automorphism of the cyclic group $G/\overline G$. Then there is necessarily an integer $a\geq 1$ such that $\chi_a$ induces $\overline \chi$ modulo $\overline G$ and $(a,|H|)=1$ \footnote{An exercise: this amounts to showing that if $b$ is an integer prime to 
$\nu = |G/\overline G|$ and $|G| =\mu \nu$, then there exists an integer $a=b+k\nu$
that is prime to $\mu \nu$. Take for $k$ the product of the prime divisors of $\mu$ 
that do not divide $b$.}. From the hypothesis, the types of $\nu(\omega)$ and $\mu(\omega)$ are the same. But so are the types of $\nu (\omega)$ and $\nu \circ \chi_a(\omega)$. Conclude that the actions $\nu \circ \chi_a$ and $\mu$ are conjugate.
\vskip 1,5mm




%

\subsubsection{Comparizon with previous forms} \label{ssec:statement_twisting_general}
We compare the general form (lemma \ref{prop:twisted cover}) with the Galois form (lemma \ref{prop:general_twisted_cover}) and the geometric monodromy group 
$S_n$ form \cite[lemma 2.1]{DeLe1}
of the twisting lemmas.
\vskip 2mm

\noindent
{\it Lemma \ref{prop:twisted cover} (general form) $\Rightarrow$ lemma \ref{prop:general_twisted_cover} (Galois form}):
%
Both forms have the assumption (const/comp). In the Galois situation from \S \ref{ssec:Galois_covers}, the $k$-cover is {\it Galois} (and so $f=g$) and the $k$-\'etale algebra is a {\it Galois} field extension 
$E/k$ with group $\Gal(E/k)=H$ (so $\prod_{l=1}^s E_l/k= E/k$ and $N=E$). 
Then statement
(i) $\Leftrightarrow$ (ii) in lemma \ref{prop:twisted cover} exactly corresponds to statement 
(iii) $\Leftrightarrow$ (v) in lemma \ref{prop:general_twisted_cover}. 

Indeed condition (ii) from lemma \ref{prop:twisted cover} reduces to its first part (ii-1) (see \S \ref{remark:twisting_general} (b))
and then coincides with condition (v) from lemma \ref{prop:general_twisted_cover}, 
and condition (i) from lemma \ref{prop:twisted cover} corresponds to condition (iii) from lemma \ref{prop:general_twisted_cover}
(note that the \'etale algebra $\prod_{l} E_l/k$ (resp. $\prod_{l} k(X)_{t_0,l}/k$) from condition (i) is a product of $|G|/|H|$ copies of the Galois field extension $E/k$ (resp. $k(X)_{t_0}/k$)).

\vskip 2mm

\noindent
{\it Lemma \ref{prop:twisted cover} (general form) $\Rightarrow$ lemma 2.1 from \cite{DeLe1}}: 
In \cite{DeLe1}, the $k$-cover $f:X\rightarrow B$ is of degree $n$ and geometric monodromy group $S_n$. Then 
$G=\overline G = S_n$, that is, we are in the standard situation (a) from \S \ref{remark:twisting_general}.
Thus condition (ii-2) holds.
%
The twisted cover $\widetilde g^N: \widetilde Z^N\rightarrow B$ in \cite[lemma 2.1]{DeLe1} is 
the twisted cover $\widetilde g^{\mu \varphi}: \widetilde Z^N\rightarrow B$ in this paper. 
Conclude that  (i) $\Rightarrow$ (ii) in  \cite[lemma 2.1]{DeLe1}  exactly corresponds to (ii) $\Rightarrow$ (i) in lemma \ref{prop:twisted cover}.

\subsubsection{Proof of the twisting lemma \ref{prop:twisted cover}} \label{ssec:proof_of_twisting_lemma}  We will use the Galois form of the twisting lemma 
to establish the general form. 
\vskip 1mm

\noindent
(i) $\Rightarrow$ (ii):
Assume (i) holds. Necessarily $N$ is the compositum of the Galois closures of the extensions $k(X)_{t_0,l}/k$. From the twisting lemma \ref{prop:general_twisted_cover} for Galois covers,
there is a unique $\gamma \in \Gamma$ satisfying condition (ii-1) from lemma \ref{prop:twisted cover}. And from lemma \ref{prop:general_twisted_cover} (a), this last condition is equivalent to existence of some $\omega \in \overline G$ such that $(\phi \circ {\sf s}_{t_0})(\tau) = \omega \hskip 2pt (\chi_\gamma \circ \varphi) (\tau) 
\hskip 2pt \omega^{-1}$ for all $\tau \in \Gabs_k$. Thus we obtain:

\vskip 0mm

$$(\nu \circ \phi \circ {\sf s}_{t_0})(\tau) = 
\nu(\omega) \hskip 3pt (\nu \circ \chi_\gamma \circ \varphi) (\tau) 
\hskip 3pt \nu(\omega)^{-1} \hskip 5mm (\tau \in \Gabs_k)$$

\noindent
But condition (i) gives $\nu \circ \phi \circ {\sf s}_{t_0} (\tau) = \beta \hskip 2pt \mu \circ \varphi(\tau)\hskip 2pt
\beta^{-1}$ $(\tau \in \Gabs_k)$, for some $\beta \in S_n$. Conjoining these equalities yields condition  (ii-2).
\vskip 1mm

\noindent
(ii) $\Rightarrow$ (i): Assume (ii) holds. From lemma \ref{prop:general_twisted_cover}, existence of $x_0\in \widetilde Z^{\chi \varphi}(k)$ such that $\widetilde g^{\chi \varphi}(x_0)=t_0$ implies that $N$ is 
the compositum of the Galois closures of the $k(X)_{t_0,l}$, and so we have  $(\phi \circ {\sf s}_{t_0})(\tau) = \omega \hskip 2pt (\chi\circ \varphi) (\tau) 
\hskip 2pt \omega^{-1}$ for some $\omega \in \overline G$ and all $\tau \in \Gabs_k$. 



Denote the orbits of $\mu: H\rightarrow S_n$, which correspond to the extensions $E_1,\ldots E_s$, 
by ${\mathcal O}_1,\ldots, {\mathcal O}_s$.
Fix one of them, {\it i.e.} $l\in \{1,\ldots,s\}$, and let $i\in \{1,\ldots, n\}$ be some index such that $E_l$ is the fixed field in $k^{\sep}$ of the subgroup 
of $\Gabs_k$ fixing $i$ {\it via} the action $\mu\circ \varphi$. 
For $j=\nu(\omega)(\sigma (i))$ (with $\sigma$ given by condition (ii-2)), we have

$$\begin{matrix}
(\nu \circ \phi \circ {\sf s}_{t_0}) (\tau) (j) & = \nu(\omega) \hskip 2pt (\nu \circ \chi \circ \varphi)(\tau) \hskip 2pt (\sigma (i)) \hfill \\ 
& =\nu(\omega) \hskip 2pt ({\rm conj}(\sigma) \circ \mu \circ \varphi)(\tau) \hskip 2pt (\sigma (i)) \hfill \\
& =\nu(\omega) \hskip 3pt \sigma \hskip 3pt  (\mu \circ \varphi)(\tau) \hskip 1pt (i) \hfill \\

\end{matrix}$$

\noindent
and so $j$ is fixed by $(\nu \circ \phi \circ {\sf s}_{t_0}) (\tau)$ if and only if $i$ is fixed by $(\mu \circ \varphi)(\tau)$. Conclude that the specialization $k(X)_{t_0,j}$ is the field $E_l$. $\square$

%


%

\penalty -3000
\section{Applications} \label{sec:varying_field}

\subsection{PAC fields} \label{ssec:field_PAC}
Recall that a field $k$ is said to be PAC if every non-empty geometrically 
irreducible $k$-variety has a Zariski-dense set of $k$-rational points. 
If $k$ is PAC, the twisting lemma leads to the following statements in the two standard 
situations (b) and (c) from \S \ref{remark:twisting_general}
(the standard situation (a) corresponds to corollary {{3.1}} from \cite{DeLe1}). Similar
applications over PAC fields can also be found in Bary-Soroker's works 
\cite{Bary-Soroker_irreducible} \cite{Bary-Soroker_Dirichlet}.

\begin{corollary} \label{cor:PAC1}
Let $k$ be a PAC field, $f:X\rightarrow B$ be a $k$-{\rm G}-Galois cover of group $G$ and 
geo\-me\-tric monodromy group $\overline G$, 
and let $E/k$ be a Galois extension of group $H\subset G$. Assume that condition 
{\rm (const/comp)} from \S \ref{ssec:Galois_covers} holds
and ${\rm Out}(G/\overline G)=\{1\}$.
Then $E/k$ is {\rm the} specialization $k(X)_{t_0}/k$ of $f$ at each point $t_0$ in a 
Zariski-dense subset of $B(k)\setminus D$.
\end{corollary}

The special case $G=\overline G$ corresponds to theorem 3.2 of \cite{DeBB} (which proved the 
Beckmann-Black conjecture over PAC fields). 

\begin{proof} Assumption ${\rm Out}(G/\overline G)=\{1\}$ assures that $\Gamma \not= \emptyset$ (remark \ref{rem:after_twisting_non-regular_galois} (a)).
Pick $\gamma \in \Gamma$. Since $k$ is PAC, the variety $\widetilde Z^{\chi_\gamma \varphi}$ has a Zarisi-dense set ${\mathcal Z}$ of $k$-rational points. From lemma \ref{prop:general_twisted_cover}, the Zariski-dense subset  $\widetilde g^{\chi_\gamma \varphi}({\mathcal Z})\setminus D \subset B(k)\setminus D$ satisfies the announced conclusion.
\end{proof}

\begin{corollary} \label{cor:PAC2}
Let $k$ be a PAC field, $f:X\rightarrow B$ be a 
degree $n$ $k$-cover 
and let $1^{\beta_1} \cdots n^{\beta_n}$ be the type of some element of the
monodromy group $G$ in the Galois representation $\nu:G\rightarrow S_n$
of $k(X)/k(B)$. Let $\prod_l E_l/k$ be an \'etale algebra such that

\noindent
- the divisor of all degrees $[E_l:k]$ is $1^{\beta_1} \cdots n^{\beta_n}$,  

\noindent
- condition {\rm (const/comp)} holds,

\noindent
- the compositum $N/k$ of the Galois closures of the extensions $E_l/k$ 
is a cyclic extension of order ${\rm ppcm}\{\hskip 2pt i\hskip 2pt|\hskip 2pt \beta_i \not=0\}$.

\noindent
Then $\prod_l E_l/k$ is 
the collection $\prod_l k(X)_{t_0,l}/k$ of specializations of $f$ at each 
point $t_0$ in a Zariski-dense subset of $B(k)\setminus D$.
\end{corollary}

A useful special case is for $1^{\beta_1} \cdots n^{\beta_n} = n^1$: it can then be concluded that $f:X\rightarrow B$ specializes to some degree $n$ field extension at each $t_0$ in a Zariski-dense 
subset of $B(k)\setminus D$ ({\it i.e.} the Hilbert irreducibility conclusion) under the assumptions
that there is a $n$-cycle in $\nu(G)$ and $k$ has a degree $n$ cyclic extension satisfying 
condition {\rm (const/comp)}. This can be compared to \cite[corollary 1.4]{Bary-Soroker_Dirichlet} 
(and \cite[corollary {{3.1}}]{DeLe1}) which has the same Hilbert conclusion under
the assumptions that $G = \overline G = S_n$ and $k$ has a degree $n$ separable extension.

\begin{proof} Let $\omega \in G$ with $\nu(\omega)$ of type $1^{\beta_1} \cdots n^{\beta_n}$. Identify the Galois group $H=\Gal(N/k)$ with the subgroup $\langle \omega \rangle \subset G$. We are in the standard situation (c) from \S \ref{remark:twisting_general} and so condition (ii-2) from lemma
\ref{prop:twisted cover} holds for some isomorphism $\chi_\gamma$ ($\gamma\in \Gamma$). 
Since $k$ is PAC, condition (ii-1) holds for all $t_0$ in a Zariski-dense subset of $B(k)\setminus D$. Therefore condition (i) from lemma \ref{prop:twisted cover} holds as well, thus ending the proof.
\end{proof}

\subsection{Finite fields} \label{ssec:finite_fields}
If $k$ is a suitably large finite field $\Ff_q$, the Lang-Weil estimates can be used to guarantee that 
the twisted covers have $\Ff_q$-rational points. More specifically we have the following result,
where we take $B=\Pp^1$ for simplicity.

\begin{corollary} \label{cor:finite-fields}
Let $f:X\rightarrow \Pp^1$ be a regular $\Ff_q$-cover of degree $n\geq 2$, with $r$ branch points and with geometric monodromy group $S_n$. Let $m_1,\ldots,m_s$ be some positive integers (possibly repeated) such that $\sum_{l=1}^s m_l = n$. Then the number ${\mathcal N}(f,m_1,\ldots,m_s)$ of unramified points $t_0 \in \Ff_q$ such that $\prod_{l=1}^s\Ff_{q^{m_l}}/\Ff_q$ is the collection of specializations of $f$ at $t_0$ can be evaluated as follows:

$$\left|{\mathcal N}(f,m_1,\ldots,m_s) - \frac{(q+1) \hskip 2pt |m_1^1 \cdots m_s^1|}{n!}\right| \leq r \hskip 1pt n! \hskip 1pt\sqrt{q}$$

\noindent
where $|m_1^1 \cdots m_s^1|$ is the number of elements in the conjugacy class in $S_n$ corresponding to the type $m_1^1 \cdots m_s^1$.
\end{corollary}

This extends similar estimates that have appeared in the literature for Galois covers under the name of  Tchebotarev theorem for function fields over finite fields. See \cite{Weil_hermann}, \cite{fried-hilbert} 
\cite{ekedahl}, \cite[\S 6]{FrJa}, and also  \cite[corollary 3.5]{DEGha2} where the Galois analog of corollary \ref{cor:finite-fields} is obtained as the outcome of our approach in the standard situation
\S \ref{remark:twisting_general} (b). 

For the type $m_1^1 \cdots m_s^1=n^1$ of $n$-cycles, we obtain that the number ${\mathcal N}(f,n)$ is asymptotic to $q/n$ when $q\rightarrow +\infty$. For example if $f:X\rightarrow \Pp^1$ over $\Ff_p$ is given by the trinomial  $Y^n+Y-T$ (which satisfies the assumptions of corollary \ref{cor:finite-fields} if $p\not| n(n-1)$ \cite[\S 4.4]{Serre-topics}), 
the number of irreducible trinomials $Y^n+Y+a\in \Ff_p[Y]$ 
realizing the extension $\Ff_{p^n}/\Ff_p$ is asymptotic to $p/n$ as $p\rightarrow \infty$,  a result  due  to Cohen \cite{cohen} and Ree \cite{ree} proving a conjecture of Chowla \cite{chowla}.

\begin{proof} 
We are in the standard situation $G=\overline G = S_n$. Condition (const/comp) trivially holds. Furthermore, it follows from the beginning note of \S \ref{remark:twisting_general} that the number of 
$\gamma \in \Gamma$ for which condition {\rm (ii-2)} holds is $1$; denote by $\chi_{0}$ 
the corresponding isomorphism. From lemma \ref{prop:twisted cover}, the set of unramified $\Ff_q$-rational points on the twisted variety $\widetilde Z^{\chi_{0} \varphi}$ maps {\it via} the cover $\widetilde g^{\chi_{0} \varphi}: \widetilde Z^{\chi_{0} \varphi}\rightarrow \Pp^1$ to the set of points $t_0\in \Pp^1(\Ff_q)$ satisfying the desired conclusion.  Using remark \ref{rem:after_twisting_non-regular_galois} (c), we obtain

$$0 \leq \frac{ | \widetilde Z^{\chi_{0} \varphi}(\Ff_q)|  }{|{\rm Cen}_{S_n}(\chi_{0}(H))|} -  {\mathcal N}(f,m_1,\ldots,m_s)  \leq \frac{ r(n!-1)  }{|{\rm Cen}_{S_n}(\chi_{0}(H))|}$$

\noindent
where $H=\Gal(\Ff_{q^M}/\Ff_q)$ with $M= {\rm ppcm}(m_1,\ldots,m_s)$ and the term $(r(n!-1))$ is an upper bound for the number of ramified points on $\widetilde Z^{\chi_{0} \varphi}$. Also note that $\widetilde g^{\chi_{0} \varphi}$ and $g$ being isomorphic over $k^\sep$, they have the same branch point number, which is the branch point number $r$ of $f$, and that the curves $\widetilde Z^{\chi_{0} \varphi}$ and $Z$ have the same genus, say ${\rm g}$.

The cyclic subgroup  $\chi_{0}(H)\subset S_n$ is generated by a permutation of type $m_1^1\cdots m_s^1$ (condition (ii-2) from lemma \ref{prop:twisted cover}). Hence we have $|{\rm Cen}_{S_n}(\chi_{0}(H))| = n! / |m_1^1 \cdots m_s^1|$. The Lang-Weil estimates give:

$$|\hskip 2pt | \widetilde Z^{\chi_{0} \varphi}(\Ff_q)| - (q+1) | \leq 2 {\rm g} \sqrt{q}$$ 

\noindent 
The Riemann-Hurwitz formula yields ${\rm g} \leq (r-2)(n!-1)/2$. The announced estimate easily follows.
(We use that the largest cardinality of a conjugacy class in $S_n$ is $n(n-2)!$, {\it i.e.}, that of the class of $n-1$-cycles; this enables us to write that $(n-1)/|{\rm Cen}_{S_n}(\chi_{0}(H))|\leq 1$).
%
\end{proof}

\subsection{Number fields} \label{ssec:number_fields}
Over number fields, we will follow a local-global approach as in \cite{DeLe1} and 
\cite{DEGha}. We start with a local result at one prime. We give two versions: a {\it mere version} for a cover $f:X\rightarrow \Pp^1$ and a G{\it -version} for a G-Galois cover $g:Z\rightarrow \Pp^1$. 


For the next two statements, let $k$ be a number field, $f:X\rightarrow \Pp^1$ be a degree $n$ regular $k$-cover, $r$ be the branch point number, $G$ (resp. $\overline G$) be the monodromy group (resp. geometric monodromy group), $g:Z\rightarrow \Pp^1$ be the Galois closure of $f$, $\nu:G\rightarrow S_n$ be 
the Galois representation of $k(X)/k(T)$ (relative to $k(Z)$) and $\widehat k_g/k$ be the constant  extension in $g$. A prime number $p$ is said to be {\it bad} if it is one from the finite list of primes for which the branch divisor is not \'etale or there is vertical ramification at $p$ \cite{DEGha}, it is said to be {\it good} otherwise.



\begin{corollary} \label{cor:fried} 
Suppose given
\vskip 0,8mm

\noindent
{\rm (in the mere version):} the type $1^{\beta_1} \cdots n^{\beta_n}$ of an element of $\nu(\overline G)\subset S_n$, 
\vskip 0,8mm

\noindent
{\rm (in the G-version):}  an element $\omega\in \overline G$.
\vskip 0,8mm

\noindent
Then for each prime $p\geq 4r^2 (n!)^2$, good and totally split in $\widehat k_g/\Qq$, there exists an integer $b_p\in \Zz$ such that for each integer $t_0\equiv b_p$ {\rm mod $p$}, 
\vskip 0,8mm

\noindent
{\rm (mere version)}
the collection of specializations of $f\otimes_k \Qq_p$ at $t_0$ is an \'etale algebra $\prod_l E_l/\Qq_p$ with degree divisor $ \prod_l [E_l:\Qq_p]^1= 1^{\beta_1} \cdots n^{\beta_n}$,
\vskip 1,5mm

\noindent
{\rm (G-version)}  the specialization of the $\Qq_p$-{\rm G}-Galois cover $g\otimes_k \Qq_p$ at $t_0$ is the unramified extension $N_p/\Qq_p$ of degree $|\langle \omega \rangle|$.

\end{corollary}

%




The mere version generalizes theorem 4 from \cite{fried-hilbert}: if $\nu(\overline G)$ contains an $n$-cycle, then, for $1^{\beta_1} \cdots n^{\beta_n} = n^1$, the conclusion, stated as in \cite{fried-hilbert} in the situation $f$ is given by a polynomial $P(T,Y)$, is that $P(t_0,Y)$ is irreducible in $\Qq_p[Y]$, and so in $k[Y]$ too.

\begin{proof} Consider first the mere version. Let $p$ be a totally split prime in the extension $\widehat k_g/\Qq$ (infinitely many such primes exist from the Tchebotarev density theorem). In particular  $\Qq_p \widehat k_g=\Qq_p$. For each $i=1,\ldots,n$ with $\beta_i>0$, let $E^{p,i}/\Qq_p$ be the unique unramified extension of $\Qq_p$ of degree $i$. 
Here we use the twisting lemma \ref{prop:twisted cover} in the ``cyclic specializations'' standard 
situation (c) from \S \ref{remark:twisting_general}; we apply it to the cover $f\otimes_k \Qq_p$ and the $\Qq_p$-\'etale algebra $\prod_{i} (E^{p,i}/\Qq_p)^{\beta_i}$,
where the exponent $\beta_i$ indicates that the extension $E^{p,i}/\Qq_p$ appears $\beta_i$ times. 
Condition (const/comp) holds by definition of $\widehat k_g$ and condition (ii-2) from lemma \ref{prop:twisted cover} holds for some isomorphism $\chi_\gamma$, $\gamma \in \Gamma$ (\S \ref{remark:twisting_general} (c)).
If $p$ is a good prime, the twisted curve $\widetilde Z^{\chi_\gamma \varphi} \otimes_k \Qq_p$ 
has good reduction, and the Lang-Weil estimates then show that if $p\geq 4r^2 (n!)^2$, the special fiber has at least one unramified $\Ff_p$-rational point; see \cite{DEGha} for more details. From Hensel's lemma, such a $\Ff_p$-rational point lifts to a $\Qq_p$-rational point on $\widetilde Z^{\chi_\gamma \varphi}$. Conclude with lemma \ref{prop:twisted cover} that the \'etale algebra $\prod_{i} (E^{p,i}/\Qq_p)^{\beta_i}$ is the collection of specializations of $f\otimes_k \Qq_p$  at each point $t_0$ in a coset of $\Zz_p$ modulo $p\Zz_p$. 


The G-version is very similar, but it is the Galois form of the twisting lemma (lemma \ref{prop:general_twisted_cover}) that should be applied, to the regular $\Qq_p$-G-Galois cover $g\otimes_k \Qq_p$ and the unramified extension of $\Qq_p$ of degree $|\langle \omega\rangle|$.
\end{proof}

Corollary \ref{cor:fried} can be used simultaneously for several types of elements in $\nu(\overline G)\subset S_n$ and for several elements of $\overline G$. The weak approximation property of $\Pp^1$ then provides arithmetic progressions $(am+b)_{m\in \Zz} \subset \Zz$ with ratio $a$ the product of several corresponding primes. In particular by using all non-trivial elements of $\overline G$, it can be  guaranteed that the specialization at $am+b$ (for every $m\in \Zz$) of the $\widehat k_g$-G-Galois cover $g\otimes_k\widehat k_g$ be a Galois extension of group $\overline G$; this uses a standard argument (recalled in  \cite[\S 3.4]{DEGha}) based on a  lemma of Jordan. This implies that the specialization at $am+b$ of the $k$-G-Galois cover $g$ is a Galois extension of group a subgroup of $G$ containing $\overline G$. As 
the $k$-cover $f:X\rightarrow \Pp^1$ is assumed to be regular (and so $\nu(\overline G)$ is a transitive subgroup of $S_n$), it follows that the collection of specializations at $am+b$ of the $k$-cover $f$ is a
single field extension of degree $n$, {\it i.e.} Hilbert's conclusion holds at $am+b$ (for every $m\in \Zz$). 

We obtain the following statement, which generalizes \cite[corollary 4.1]{DeLe1} to arbitrary regular covers. 

The constants however are not as good as in the ``$G=\overline G=S_n$'' situation of \cite{DeLe1} because of the preliminary condition on the primes $p$ that uses the Tchebotarev 
theorem.

\begin{corollary} \label{cor:effective} 
There exist integers 
$m_0, \beta>0$ depending on $f$ such that the following holds.
Let ${\mathcal S}$ be a finite set of primes $p>m_0$, good and totally split in $\widehat k_g/\Qq$, 
each given with  positive integers $d_{p,1}\ldots, d_{p,s_p}$ (possibly repeated) such that 
$d_{p,1}^1\cdots d_{p,s_p}^1$ is the type of some element in $\nu(\overline G)$.
Then there exists an integer $b\in \Zz$ such that 
\vskip 1,2mm

\noindent
{\rm (i)} $0\leq b\leq \beta \prod_{p\in {\mathcal S}} p$,  

\vskip 1,2mm

\noindent
{\rm (ii)} for each integer $t_0\equiv b$ {\rm mod $(\beta \prod_{p\in {\mathcal S}} p)$} , $t_0$ is not a branch point of $f$ and the collection of specializations of $f$ at $t_0$ is a single degree $n$ field extension with residue degrees $d_{p,1}\ldots, d_{p,s_p}$ at $p$ for each $p\in {\mathcal S}$.
\end{corollary}

\noindent
{\it Addendum} \ref{cor:effective} (on the constants) 
Denote 
the number of non-trivial conjugacy classes of $\overline G$ by $\hbox{{\rm cc}}(\overline G)$.
One can take $m_0$ such 
that the interval $[4r^2 (n!)^2,m_0]$ contains at least $\hbox{{\rm cc}}(\overline G)$ primes, good and totally split in $\widehat k_g/\Qq$, and $\beta$ to be the product of $\hbox{{\rm cc}}(\overline G)$ such 
primes. 
\medskip

\begin{proof} We use corollary \ref{cor:fried} simultaneaously for several primes: a first set of primes associated to all non-trivial elements of $\overline G$ as explained above, and the set of primes given
 in the statement with the associated types. We apply the G-version of corollary \ref{cor:fried} to the former data and the mere version to the latter. This provides an arithmetic progression $(am+b)_{m\in \Zz} \subset \Zz$ with ratio $a=\beta \prod_{p\in {\mathcal S}} p$ where $\beta >0$ is the product of the primes in the first set. The primes dividing $\beta$ guarantee that the collection of specializations at $am+b$ of  the $k$-cover $f$ is a single field extension $E/k$ of degree $n$. And each of the primes $p\in {\mathcal S}$ gives that the $\Qq_p$-\'etale algebra $E\otimes_k\Qq_p$ has degree divisor $d_{p,1}^1\cdots d_{p,s_p}^1$.\end{proof}

\subsection{Ample fields} \label{ssec:ample_fields}
Recall that a field $k$ is said to be {\it ample} if every smooth $k$-curve with a $k$-rational point has infinitely many $k$-rational points. Over an ample field, the twisting lemma \ref{prop:twisted cover} yields the following statement which generalizes \S 3.3.2 (***) from \cite{DeBB}.

\begin{corollary} \label{cor:ample}
Let $k$ be an ample field and $f:X\rightarrow B$ be a \hbox {degree $n$}  $k$-cover of curves. Let $t_0\in B(k)$ not in the branch point set {\bf t}. There exist infinitely many $t\in B(k)\setminus {\bf t}$ such that the collection of specializations $\prod_{l} k(X)_{t,l}/k$ and $\prod_{l} k(X)_{t_0,l}/k$ at $t$ and $t_0$ respectively are equal.
\end{corollary}

\begin{proof}
Take the $k$-\'etale algebra $\prod_{l=1}^s E_l/k$ from lemma \ref{prop:twisted cover} to be the  collection $\prod_{l=1}^s k(X)_{t_0,l}/k$ of specializations  at $t_0$. With the notation from \S \ref{ssec:Galois_covers}, we have $\varphi = \phi \circ {\sf s}_{t_0}$ and $\overline \varphi = \Lambda$.
Hence condition (const/comp) holds with $\overline \chi = {\rm Id}_{G/\overline G}$, and $\Gamma \not=\emptyset$ (remark \ref{rem:after_twisting_non-regular_galois} (a)). 
From implication (i) $\Rightarrow$ (ii) in lemma \ref{prop:twisted cover}, there exists $\gamma \in \Gamma$ such that conditions (ii-1) and (ii-2)
are satisfied for $t_0$ with $\chi=\chi_\gamma$. Condition (ii-1) is that there exists $x_0\in \widetilde Z^{\chi \varphi}(k)$  with $\widetilde g^{\chi \varphi}(x_0)=t_0$. As $k$ is ample and $\widetilde Z^{\chi \varphi}$ is a smooth $k$-curve, there are infinitely many $k$-rational points $x$ on $\widetilde Z^{\chi \varphi}$. The corresponding points $t= \widetilde g^{\chi \varphi}(x)\in B(k)$, excluding the branch points, satisfy conditions (ii-1) and (ii-2) from lemma \ref{prop:twisted cover}. Implication (ii) $\Rightarrow$ (i) of this lemma finishes the proof.
\end{proof}

\begin{remark} The proof and the result generalize to higher dimensional covers $f:X\rightarrow B$. It should be assumed however that the covering space $Z^\sep$ of the cover $Z^\sep\rightarrow B\otimes_kk^\sep$ corresponding to the field extension $k^\sep(Z)/k^\sep(B)$ is smooth ($Z^\sep$ is 
the normalization of $B$ in the field $k^\sep(Z)$ (defined in \S \ref{ssec:basic2})  and so is {\it a priori} only normal).
The ampleness of $k$ then provides a Zariski-dense subset of $k$-rational points on 
$\widetilde Z^{\chi \varphi}$ and the conclusion becomes that there exists a Zariski-dense subset 
${\mathcal B} \subset B(k)\setminus D$ such that the collection of specializations 
$\prod_{l} k(X)_{t,l}/k$ at each $t\in {\mathcal B}$ equals $\prod_{l} k(X)_{t_0,l}/k$.


\end{remark}

\bibliography{FCAmore2}
\bibliographystyle{alpha}

\end{document}